\documentclass[a4paper,11pt]{article}

\input{styleart.sty}
\usepackage{xcolor}
\usepackage{leftindex}
\def\<#1>{\mathinner{\langle#1\rangle}}

\usepackage{mathtools} 
\usepackage{amsmath} 
\usepackage{qtree}
\usepackage{forest}

\newcommand{\cspace}{~}

\usepackage{todonotes}

\title{Model theory of term algebras revisited}
\date{}
\author{Davide Carolillo, Yifan Jia, Bakh Khoussainov \& Rizos Sklinos} 

\begin{document}

\maketitle

\begin{abstract}
Building on work of Maltsev \cite{Malcev71} on locally free algebras in finite purely functional languages, we revisit the model theory of (absolutely free) term algebras and their completions. Maltsev's analysis yields a natural axiomatization together with quantifier elimination to positive Boolean combinations of special formulas, and shows that the complete extensions are parametrized exactly by the number \(k\in\{0,1,\dots,\omega\}\) of indecomposable elements; for \(1\le k\le\omega\) the standard model is the free term algebra on \(k\) generators. 

We give a new, quantifier-elimination--free proof of completeness using Ehrenfeucht--Fra\"iss\'e games, and we establish several further structural properties of the standard models and theories. In particular, for \(1\le k\le\omega\) we prove first-order rigidity and atomicity of the standard model. 

For every \(0\le k\le\omega\) we show that the corresponding theory does not have the finite cover property and weakly eliminates imaginaries. We also provide new proofs of stability-theoretic features previously obtained by Belegradek \cite{Belegradek1988LocallyFree}: the theories are stable but not superstable, normal (hence \(1\)-based), and have trivial forking; consequently, no infinite group is interpretable in any model. 

Finally, we analyze model completeness and show that \(T_0\) is the model companion of the theory of locally free algebras, while the theories with \(k\ge 1\) are not model complete.
\end{abstract}

\section{Introduction} 
Maltsev proved \cite{Malcev71} that, in any fixed finite purely functional language, the class of locally free algebras admits a natural axiomatization and elimination of quantifiers down to positive Boolean combinations of \emph{special formulas} (see Section~\ref{Prel}). In particular, his analysis yields a clean description of the completions of this axiomatization: they are parametrized exactly by the number of indecomposable elements. More precisely, for each cardinal \(0\leq k\leq \omega\) there is a distinct complete theory extending the axioms of locally free algebras, obtained by stipulating that the algebra has precisely \(k\) indecomposables; and these are the only complete extensions.

For every \(1\leq k\leq \omega\), the canonical (standard) model of the resulting complete theory is the absolutely free term algebra on \(k\) generators. Moreover, as a straightforward consequence of the same quantifier-elimination framework, each of these theories---including the case \(k=0\)---is stable.

These theories were subsequently investigated from the point of view of stability theory by Belegradek \cite{Belegradek1988LocallyFree,Belegradek1989ForkingLocallyFree}. In particular, he observed that passing to a suitable definitional expansion by \emph{projections} yields full quantifier elimination, and he exploited this strengthening as a technical lever for a finer stability-theoretic analysis. Unfortunately, these results have not (to our knowledge) been translated into English; as a consequence, we were unaware of this reference during the period in which the bulk of the present work was completed.

The present paper is primarily a guided survey of the existing landscape, filtered through our own perspective, but it also contains several new results. Most notably, we employ Ehrenfeucht--Fra\"iss\'e games to reprove completeness of the theories above, entirely independently of quantifier elimination. We also establish the following results.

\begin{theorem}
For \(1\leq k\leq \omega\), the standard model \(\mathcal{F}_k\) is first-order rigid.
\end{theorem}

We further show:

\begin{theorem}
For each \(1\leq k\leq \omega\), the standard model \(\mathcal{F}_k\) is atomic.
\end{theorem}

In contrast, no locally free algebra with no indecomposable elements is atomic.

The next two statements also appear to be new.

\begin{theorem}
For each \(0\leq k\leq \omega\), the theory \(\mathcal{T}_k\) does not have the finite cover property.
\end{theorem}

We also prove:

\begin{theorem}
For each \(0\leq k\leq \omega\), the theory \(\mathcal{T}_k\) weakly eliminates imaginaries.
\end{theorem}

In addition, we provide new proofs of the following facts, which were already obtained in \cite{Belegradek1988LocallyFree}.

\begin{theorem}
For each \(0\leq k\leq \omega\), the theory \(\mathcal{T}_k\) is not superstable.
\end{theorem}

Moreover, we formulate a general criterion for non-superstability which appears to apply in a broad range of analogous contexts.

\begin{theorem}
For each \(0\leq k\leq \omega\), the theory \(\mathcal{T}_k\) is normal, and hence \(1\)-based.
\end{theorem}

\begin{theorem}
For each \(0\leq k\leq \omega\), the theory \(\mathcal{T}_k\) has trivial forking. In particular, no infinite group is interpretable in any model of the theory.
\end{theorem}

\begin{theorem}
For each \(1\leq k\leq \omega\), the theory \(\mathcal{T}_k\) is not model complete.
\end{theorem}

\begin{theorem}
The theory \(T_0\) is the model companion of the theory of locally free algebras.
\end{theorem}

Finally, we expect the paper to have some pedagogical value for readers approaching model theory for the first time: locally free algebras provide a concrete arena in which a number of fundamental notions and dividing lines can be seen ``in the wild.'' With this goal in mind, we have occasionally expanded the exposition to include background material that will be familiar to most model theorists, but may be helpful to newcomers.

\section{Preliminaries}\label{Prel}

In this section we define term algebras and record some of their basic properties. Fix a purely functional signature (with constants)
\[
\Sigma=\{\,f_0,\ldots,f_k,\; c_0,c_1,\ldots\,\},
\]
where each $f_i$ is a function symbol of arity $n_i$ and the $c_j$ are constant symbols. We assume that $\Sigma$ contains at least one constant symbol and that at least one of the function symbols has arity $\geq 2$. When $\Sigma$ has exactly $n\in\mathbb{N}$ constant symbols, we sometimes write $\Sigma_n$ in place of $\Sigma$; if $\Sigma$ has infinitely many constants, we may write $\Sigma_\omega$.

\subsection{Terms and ground terms}
The set of \emph{terms} over $\Sigma$ is defined inductively as follows:
\begin{itemize}
\item Every variable $x$ and every constant symbol $c$ is a term.
\item If $t_1,\ldots,t_\ell$ are terms and $f$ is a function symbol of arity $\ell$, then $f(t_1,\ldots,t_\ell)$ is a term.
\end{itemize}
In this paper we will only consider \emph{ground terms}, i.e.\ terms with no variables. Let $F$ denote the set of all ground terms. We view $F$ as the universe of a $\Sigma$-structure $\mathcal{F}$ by interpreting each function symbol $f\in\Sigma$ of arity $\ell$ via
\[
f^{\mathcal{F}}(t_1,\ldots,t_\ell)\ :=\ f(t_1,\ldots,t_\ell)
\qquad(t_1,\ldots,t_\ell\in F).
\]
To simplify notation, we write $f$ in place of $f^{\mathcal{F}}$ throughout. Thus,
\[
\mathcal{F}\;=\;\bigl(F;\ f_0^{n_0},\ldots,f_k^{n_k},\ c_0,c_1,\ldots \bigr).
\]
The algebra $\mathcal{F}$ is generated by the constants $c_0,c_1,\ldots$. If $\Sigma$ has exactly $n\in\mathbb{N}$ constant symbols we write $\mathcal{F}_n$ for $\mathcal{F}$; if $\Sigma$ has infinitely many constants we may also write $\mathcal{F}_\omega$.

The algebras $\mathcal{F}_n$ are called \emph{term algebras} (or \emph{ground term algebras}). They are \emph{absolutely free} in the following sense: for every $\Sigma$-algebra $\mathcal{A}$ generated by the values of the constant symbols, there exists a surjective homomorphism $\mathcal{F}\twoheadrightarrow \mathcal{A}$.

Let $Alg(\Sigma)$ denote the class of all $\Sigma$-algebras generated by the values of the constant symbols. If $\mathcal{B}\in Alg(\Sigma)$ has the property that for every $\mathcal{A}\in Alg(\Sigma)$ there exists a surjective homomorphism $\mathcal{B}\twoheadrightarrow \mathcal{A}$, then $\mathcal{B}\cong\mathcal{F}$. Hence $\mathcal{F}$ is, up to isomorphism, the unique universal object in $ Alg(\Sigma)$.

\subsection{Reducts}
We may remove the constant symbols from the language and consider the reduct of $\mathcal{F}$ to the function symbols only. By an abuse of notation, we again denote this reduct by $\mathcal{F}$. Accordingly, for $m\in\omega$ the algebras $\mathcal{F}_m$ are finitely generated. Note that all these reducts have the same purely functional signature
\[
\Sigma^{r}=(f_0^{n_0},\ldots,f_k^{n_k}),
\]
which, by a further abuse of notation, we will continue to denote by $\Sigma$.

\begin{proposition}\label{prop:term-algebras-iso-elem}
Let $1\leq m,n\leq\omega$. Then the term algebras $\mathcal{F}_m$ and $\mathcal{F}_n$ are isomorphic if and only if they are elementarily equivalent.
\end{proposition}

\begin{proof}
Assume first that $m,n\in\omega$. The algebra $\mathcal{F}_m$ satisfies the first-order sentence asserting that \emph{there exist exactly $m$ elements $x_1,\ldots,x_m$ such that none of them is $f$-decomposable for any $f\in\Sigma$}. This sentence fails in $\mathcal{F}_n$ whenever $n\neq m$. Hence $\mathcal{F}_m\not\equiv\mathcal{F}_n$ for $m\neq n$, and the forward implication follows.

Conversely, if $\mathcal{F}_m\equiv\mathcal{F}_n$, then necessarily $m=n$ by the previous paragraph, and thus $\mathcal{F}_m\cong\mathcal{F}_n$.
\end{proof}

\begin{corollary}\label{rigid}
Let $1\leq m<\omega$. Then $\mathcal{F}_m$ is first-order rigid: if $\mathcal{A}$ is finitely generated and $\mathcal{A}\models Th(\mathcal{F}_m)$, then $\mathcal{A}\cong\mathcal{F}_m$.
\end{corollary}


One of the useful concepts is the notion of heights of the terms:
\begin{definition}
Let $t$ be a term. Then 
\begin{itemize}
\item if $t$ is a variable, say $t=x$, then  $height(t)=0$;
\item if $t=f(t_1, \ldots, t_k)$,   then $height(t)=max\{height(t_1),\ldots,height(t_k)\}+1$. 
\end{itemize}
\end{definition}
\subsection{Maltsev's Quantifier Elimination}
Maltsev proved a quantifier elimination result which we next present.

It is easy to see that one can express in first-order logic that there exist at least $m$ indecomposable elements, we call this sentence $E^m$. Apparently $D^m:=\lnot E^{m+1}$ expresses that there are at most $m$ indecomposable elements. Finally, $N_i(x)$ is the first-order formula that expresses that $x$ is not $f_i$-decomposable.   

Following Maltsev we give the following definition. 

\begin{definition}
A first-order formula $\phi(\bar x)$ in the signature $\Sigma$ is called standard, if it has the following form. 
\begin{equation}\label{specialformulas}
			\exists y_1,\ldots,y_m\left(\bigwedge_{i\in I} x_i=t_i\wedge \bigwedge_{j\in J}x_j\neq u_j \wedge\bigwedge_{h\in H}y_h\neq v_h\wedge \bigwedge_{r\in R}\bigwedge_{s\in S_r} N_s(y_r)\right)
		\end{equation}
\end{definition}

\begin{fact}\cite[Theorem 3]{Malcev71}\label{MalcevQE}
In the class of locally free $\Sigma$-algebras, every first-order formula in the signature $\Sigma$ is equivalent to a positive Boolean combination of formulas in the following families: standard formulas, $D^m$, $E^m$, and $N_i(x)$. 
\end{fact}

The above result implies that any two locally free algebras with the same finite number of indecomposable elements or any two such algebras with infinitely many indecomposable elements (independent of cardinality) are elementarily equivalent. We will give a more elementary proof of this corollary by using Ehrenfeucht-Fraisse games (see Theorem \ref{IndecomposableEquivalentElementary}). 

As a side note, we observe that the first-order theory of $\mathcal{F}_\omega$ as well as the first-order theory of locally free algebras without indecomposable elements do not admit finitely generated models. In the first case a finitely generated model would be a free algebra of finite rank, and in the second case we would have that a generator, say $a_1$ should be decomposed as $f(\bar a)$ where $\bar a$ cannot include $a_1$, continuing repeatedly choosing an element from $\bar a$ we will run out of elements we can use in finitely many steps. Hence Corollary \ref{rigid} does not apply to these theories.

\section{Model theory}\label{section 3}

Recall that we have fixed a (purely functional) signature $\Sigma:=\{f_0, \ldots, f_k\}$. We fix a cardinal $\kappa$ and, following Maltsev, we give a list of natural axioms for the first-order theory of the term algebra of rank $\kappa$. Maltsev used his quantifier elimination result in order to prove that this list indeed axiomatizes the first-order theory of such algebra.  We, on the other hand, use Ehrenfeucht-Fra\"iss\'e games. In particular we will show that all term algebras of infinite rank have the same first-order theory. In contrast, for any two finite cardinals $m,n$, the first-order theories of $\mathcal{F}_m$ and $\mathcal{F}_n$ differ.

Let $\mathcal{F}_\kappa$ be the term algebra of rank $\kappa$. Then it satisfies the following axioms:

For each $0\leq i\leq k$, we have:

\begin{equation}
\tag{Axiom $A_i$}
\forall \bar x,\bar y \big(f_i(\bar x)=f_i(\bar y)\rightarrow (x_1=y_1\land\ldots\land x_{n_i}=y_{n_i})\big)
\end{equation}

For every $0\leq i<j\leq k$, we have:

\begin{equation}
\tag{Axiom $B_{ij}$}
\forall \bar x,\bar y \big(f_i(\bar x)\neq f_j(\bar y)\big)
\end{equation}

Finally for every term $t(x,x_1,\ldots,x_n)$, in variables $x,x_1,\ldots,x_n$, which is different from $x$ and $x$ occurs in it, we have: 

\begin{equation}
\tag{Axiom $C_t$}
\forall  x, x_1,\ldots, x_n \big(t(x,x_1,\ldots,x_n)\neq x\big)
\end{equation}

It is not hard to prove that the above (infinite due to Axioms $C_t$) list of axioms is complete for the class of locally (absolutely) free algebras, i.e. algebras for which every finitely generated subalgebra is (absolutely) free (see \cite[Chapter 23: Theorem 1]{Malcev71}). In addtion, Maltsev also proved that the previously mentioned class is not finitely axiomatizable. 

In order to axiomatize the first-order theory of $\mathcal{F}_m$, for $m$ finite (including $0$), one needs to add an axiom that says that there exist exactly $m$ "indecomposable" elements. For the sake of clarity we first define the following first-order formulas. 
An element $a$ in some algebra $\mathcal{A}$ is {\em $f_i$-indecomposable}, if $\mathcal{A}\models\forall y_1,\ldots, y_{n_i}\big(a\neq f_i(y_1,\ldots,y_{n_i})\big)$. Let $N_i(x):=\forall y_1,\ldots, y_{n_i}\big(x\neq f_i(y_1,\ldots,y_{n_i})\big)$ and
\begin{equation}\label{eq.1 - def - N(x)}
    N(x):=N_0(x)\land N_1(x)\land\ldots\land N_k(x).
\end{equation}
\noindent
An element $a$ in some algebra $\mathcal{A}$ is  {\em indecomposable} if $\mathcal{A}\models N(a)$.

\begin{equation}
\tag{Axiom $D_m$}
\begin{split}
\exists x_1,\ldots, x_m \big(\bigwedge_{1\leq i<j\leq m} (x_i\neq x_j)  \land 
N(x_1)\land\ldots N(x_m)\big) \bigwedge \\
\forall x_1,\ldots, x_m, x_{m+1}\big(\lnot N(x_1)\lor \ldots \lor \lnot N(x_m)\lor \lnot N(x_{m+1})\big)
\end{split}
\end{equation}

If $\kappa$ is an infinite cardinal, then in order to axiomatize the first-order theory of $\mathcal{F}_\kappa$ we need to add an infinite list of axioms, for each $r<\omega$ we add:  

\begin{equation}
\tag{Axiom $E_r$}
\exists x_1,\ldots, x_r \big(\bigwedge_{1\leq i<j\leq r} (x_i\neq x_j)  \land 
N(x_1)\land\ldots N(x_r)\big)
\end{equation}

We will prove the next theorem using (unnested) Ehrenfeucht-Fra\"iss\'e games. We follow \cite[Section 3.3]{HodgesBook}. We first develop some terminology that will facilitate the proof. 

\begin{definition}
Let $a$ be an element in a (locally free) algebra $\mathcal{A}$. The $n$-skeleton of $a$, denoted $sk_n(a)$, is the term of height $n$, $t(\bar x)$, such that there is $\bar b$ in $A$ with $t(\bar b)=a$ and $\bot$ otherwise. 

Moreover, the $n$-downward closure of $a$, denoted $dc_n(a)$, is the set of elements of the tuple $\bar b$ in case $sk_n(a)\neq\bot$ and $\emptyset$ otherwise. 

We finally denote $dc_{\leq n}(a)=\bigcup_{0\leq m\leq n} dc_n(a)$. We call this the $n$-comprehensive downward closure. We sometimes omit $n$ when it is clear from the context. 
\end{definition}

\begin{remark}
It follows from the axioms of locally free algebras that the above notions are well defined (up to ignoring the "names" of the variables). 
\end{remark}

We bring an example for the convenience of the reader. 
\begin{example}
Let $a\in \mathcal{A}$. The $0$-skeleton of $a$ is $sk_0(a)=x$ and the $0$-downward closure is  $dc_0(a)=a$. On the other hand, if $a=f(b_1, h(b_2, b_3),g(b_4, b_5))$, then $sk_2(a)=f(x_1, h(x_2,x_3), g(x_4,x_5))$, $dc_2(a)=\{b_1, b_2, b_3, b_4, b_5\}$, and $dc_{\leq 2}(a)=\{b_1, b_2, b_3, b_4, b_5, h(b_2,b_3),$ $g(b_4, b_5), a\}$. In the case all the $b_i$'s are indecomposable we get that $sk_3(a)=\bot$.
\end{example}

\begin{theorem}\label{IndecomposableEquivalentElementary}
Let $0\leq m<\omega$. Then the following list of axioms $(A_i)_{i\leq k}$, $(B_{ij})_{0\leq i<j\leq k}$, $(C_t)_{t\neq x}$ and $D_m$ is complete. 
\end{theorem}
\begin{proof}
We will show that if two $\Sigma$-structures, $\mathcal{A}$ and $\mathcal{B}$, satisfy the axioms of the hypothesis, they are elementarily equivalent. Equivalently, we need to show that for each $r<\omega$, Player II has a winning strategy in the $r$-round game $EF_r[\mathcal{A}, \mathcal{B}]$. We recall that unnested atomic formulas in $\Sigma$ are formulas of the form, $x=y, f_0(x_1, \ldots, x_{n_0})=y, \ldots, f_k(x_1,\ldots, x_{n_k})=y$. The strategy is as follows: first of all, by Axiom $D_m$, the structures $\mathcal{A}$ and $\mathcal{B}$ have the same number of indecomposable elements, namely $m$. We choose a "random" correspondence between the two sets of indecomposable elements, unless $m=0$, in which case this first step is not necessary. 

At every round of the game Player II copies the skeleton (of depth that depends on the round) of the choice of Player I (independent of which structure player I chooses) and generates an element in the respective structure by choosing the generating elements according to the isomorphism of the indecomposable elements but also the "inventory" (to be defined) of the previous round.

We define the strategy recursively:
\begin{itemize}
\item At round 1, suppose Player I plays element $a_1$ from $\mathcal{A}$. Player II copies $sk_r(a_1)$ and checks $dc_r(a_1)$ (if $sk_r(a_1)=\bot$, then it copies $sk_l(a)$ for the largest $l\leq r$ for which $sk_l(a_1)\neq\bot$). For any element in $dc_r(a_1)$ which is indecomposable Player II will replace it with the corresponding indecomposable element in $\mathcal{B}$. On the other hand,  for any decomposable element in $dc_r(a_1)$, Player II will choose at random a decomposable element in $\mathcal{B}$. 
At this point, we extend the isomorphism between the indecomposable elements to include the corresponding choices of this round among all elements between $dc_{\leq r}(a_1)$ and $dc_{\leq r}(b_1)$, where $b_1$ is the actual answer of Player II. The inventory of the first round $I_1$ is the data that consists of: $dc_{\leq r}(a_1), dc_{\leq r}(b_1)$ the isomorphism between the indecomposable elements and the isomorphism between $dc_{\leq r}(a_1)$ and $ dc_{\leq r}(b_1)$. Note that the strategy would be the same if Player I had chosen an element from $\mathcal{B}$. Also, it is important that we keep in the inventory all $dc_{\leq r}(a_1)$ and $ dc_{\leq r}(b_1)$ and not just $dc_{r}(a_1)$ and $ dc_{r}(b_1)$.

\item At round $i+1$, suppose Player I plays element $a_{i+1}$ from $\mathcal{A}$. Player II copies the skeleton $sk_{r-(i+1)}(a_i)$ (or the largest skeleton less than $r-(i+1)$ which is not $\bot$) and checks $dc_{r-(i+1)}(a_i)$, he then uses the inventory of the previous move, $I_i$, and chooses generating elements accordingly. The new inventory, $I_{i+1}$, is $I_i$ together with the new elements in $dc_{\leq r-(i+1)}(a_i)$, $dc_{\leq r-(i+1)}(b_i)$, where $b_i$ is the actual answer of Player II for this round, and the isomorphism between them.
\end{itemize}

We now prove that this strategy works, i.e. it is indeed a winning strategy for Player II.

For a warm up we start with the unnested formula $x=y$. Suppose $a_i=a_j$, and without loss of generality we may assume that $i>j$. Then, since the inventory of the $i-1$-round contains $a_j$ and the comprehensive downward closure of an element in later rounds is contained in the comprehensive downward closure of the element in earlier rounds (since the depth decreases as we move forward in the game), we get that $b_i$ admits a construction as $b_j$ (independent of whether Player I played $b_i$ or $a_i$). Hence $b_i=b_j$.   

We move to the case $y=f(x_1, \ldots, x_n)$. Suppose $a_j=f(a_{i_1}, a_{i_2}, \ldots, a_{i_n})$, for some $j\leq r$ and $f$ some function symbol of arity $n$ from $\Sigma$. Notice that $a_j$ cannot appear in $f$. Also for any $a_i$, with $i<j$, $a_i$ belongs to the inventory of the $j-1$-th round and since $sk_1(a_j)=f(x_1,\ldots,x_r)$, we have that $sk_1(b_j)=f(x_1, \ldots,x_r)$. In addition, as before, the comprehensive downward closure of an element in later rounds is contained in the comprehensive downward closure of the element in earlier rounds (since the depth decreases as we move forward in the game). This implies that if $a_1$, for example, appears in $f$, say in the $l$-th argument, then $b_1$ must appear in the $l$-th argument of $f$ in the construction of $b_j$ as an image of $f$ (independent of whether Player I played $a_j$ or $b_j$). 

We, thus, need to show that if $a_i$, for $i>j$, appears in $f$, then $b_i$ must also appear in the same argument of $f$ in the corresponding construction of $b_j$. Suppose $a_i$ is the $l$-th argument of $f$ (it may appear in more than one places but this will not make any difference). The $j-1$ comprehensive downward closure of $a_i$ will be contained in the $I_j$ inventory and it will have a corresponding element $b'$ from the structure $\mathcal{B}$. We now take cases according to whether Player I plays $a_i$ or $b_i$. Assume first that in the $i$-round Player I plays element $a_i$ from $\mathcal{A}$, then $a_i$ is determined (since it appears in the inventory of a previous round) i.e. its construction does not require any free choices. In particular, $b_i$ will be constructed using only elements from the inventory and thus, it must be equal to $b'$ which is the $l$-th argument of $f$ as we wanted. Next assume that Player I plays element $b_i$ from $\mathcal{B}$ and assume $b_i$ is not equal to $b'$. On the other hand, the $(r-i)$-skeleton of $b_i$ and $a_i$ are the same and the $(r-i)$-downward closure of $a_i$ consists of elements in correspondence with the $(r-i)$-downward closure of $b_i$. We claim that  the $(r-i)$-skeleton of $b_i$ must be the same as the $(r-i)$-skeleton of $b'$. Indeed, both must be equal to the $(r-i)$-skeleton of $a_i$ which is equal to the $(r-i)$-skeleton of the $l$-th argument of $f$ in the analysis of $a_j$. Hence, there must be an element in the $(r-i)$-downward closure of $b_i$ that differs from that of $b'$. In particular, since there is a correspondence between the elements of the $(r-i)$-downward closure of $b'$ and those of $a_i$ (given in the $I_j$ inventory), we must have that there is an element in the construction of $a_i$ in the $i$-th move that differs from its construction as the $l$-th argument of $f$ in the analysis of $a_j$ in the $j$-th round, in particular $a_j\neq f(a_{i_1}, a_{i_2}, \ldots, a_{i_n})$, a contradiction.    
    
Notice that it is crucial in the strategy that Player II copies the skeleton of lower and lower depth as the game moves forward. Indeed, if independent of the round Player II kept copying the $r$-skeleton of the choice of Player I, then it can be easily seen that Player I may win.  

\end{proof}

\begin{example}
We bring two examples that show the necessity of choosing the details of the above strategy:   
\ \begin{itemize}
\item First assume that in the inventory we only collect the downward closures and not the comprehensive downward closures. Consider a $3$ round game. In the first round Player I plays $a_1=g(g(f(c_1, c_2)))$ from $\mathcal{A}$, hence Player II replies with $b_1=g(g(f(d_1,d_2)))$, with $d_i$ indecomposable if and only if $c_i$ is. The $3$-downward closures are $\{c_1, c_2\}$ and $\{d_1, d_2\}$, while the comprehensive closures are $\{c_1, c_2, f(c_1, c_2), g(f(c_1, c_2)), g(g(f(c_1, c_2))) \}$ and $\{d_1, d_2, f(d_1,d_2), g(f(d_1,d_2)), g(g(f(d_1,d_2)))\}$. In the second round Player I chooses some indecomposable element and Player II does likewise. In the third round Player I chooses $a_3=g(f(c_1,c_2))$ and the $1$-downward closure of $a_3$ is $\{f(c_1,c_2)\}$ which is an element that does not belong to the inventory, hence Player II chooses randomly an element not in the inventory, in particular he might well choose an element not equal to $f(d_1, d_2)$. But then $g(a_3)=a_1$ while $g(b_3)$ is not necessarily equal to $b_1$.
\item The next example shows that we necessarily need to decrease the depth of the skeleton that Player II copies as the game move forward. So, assume we always choose to copy the $r$-skeleton independent of the round. Consider a $2$-round game. In the first round Player I plays $a_1=g(g(c))$ from $\mathcal{A}$, where $c=f(c_1, c_2)$ and the $c_i$'s are decomposable, then Player II will reply with $b_1=g(g(d))$ where $d$ is some "random" indecomposable element from $\mathcal{B}$. In the second round Player I plays $a_2=g(c)$ and Player II replies with $g(f(d_1, d_2))$, where $d_1, d_2$ are chosen arbitrarily. In particular, it might be the case that $f(d_1, d_2)$ is not $d$. Thus $a_1=g(a_2)$, but $b_1$ is not necessarily $g(b_2)$.      
\end{itemize}
\end{example}

The only reason we separated the following statement is notational. The same proof works.

\begin{theorem}
Let $\kappa$ be an infinite cardinal. Then the following list of axioms $(A_i)_{i\leq k}$, $(B_{ij})_{0\leq i<j\leq k}$, $(C_t)_{t\neq x}$ and $(E_r)_{r<\omega}$ is complete. 
\end{theorem}

The value of the above propositions is that one can prove the competeness of the above theories without appealing to the quantifier elimination of Maltsev. 

We next consider term algebras with respect to some basic model theoretic properties, such as homogeneity, atomicity, saturation, etc  

\begin{proposition}
The term algebra of rank $\omega$, $\mathcal{F}_\omega$, is atomic. 
\end{proposition}
\begin{proof}
Let $\mathcal{F}_\omega=\langle e_1, e_2, \ldots, e_m, \ldots \rangle$. The formula $B_m(x_1, \ldots, x_m)$ that says $x_i$'s are pairwise distinct indecomposable elements is $\emptyset$-definable. Note that the solution set of $B_m(\bar x)$ in $\mathcal{F}_\omega$ is the orbit of $e_1, \ldots, e_m$ under $Aut(\mathcal{F}_\omega)$. For any finite tuple $\bar a$ in $\mathcal{F}_\omega$, we consider the tuple of terms $\bar t(\bar x)$ such that, for $m$ large enough, $\bar t(e_1, e_2, \ldots, e_m)=\bar a$. We next define the formula $\phi(\bar x):=\exists \bar y \big( B_m(\bar y) \land \bar x=\bar t(\bar y)\big)$. This formula clearly isolates $tp^{\mathcal{F}_\omega}(\bar a)$. 
\end{proof}

The above proof works for any finitely generated term algebra as well. 

\begin{proposition}
Let $1\leq \kappa \leq \omega$. The term algebra of rank $\kappa$ is atomic.
\end{proposition}

As a matter of fact, for $1\leq m<\omega$, every realized type in $\mathcal{F}_m$ is algebraic with exactly as many solutions as permutations of $\{e_1, \ldots, e_m\}$.

The next result follows from abstract model theory.

\begin{corollary}
Any term algebra of rank $1\leq \kappa\leq \omega$ is homogeneous and prime model of its theory.
\end{corollary}

\begin{proposition}\label{UncountableTypes}
Any completion with indecomposable elements of the theory of locally free algebras admits uncountably many types over the empty set. 
\end{proposition}
\begin{proof}
For notational clarity we will assume that $f:=f_0$ is a $2$-ary function symbol. We will construct a rooted infinite complete binary tree where we will assign at each node a first-order formula. The collection of formulas that correspond to a branch of the tree, i.e. for  $\eta\in 2^{\omega}$, $\{\phi_{\eta\upharpoonright i}(x) \ | \ i<\omega\}$ is consistent. But any two branches will be inconsistent, i.e. if $\eta, \tau\in 2^{\omega}$ and $i$ is the smallest natural number such that $\eta(i)\neq \tau(i)$, then $\phi_{\eta\upharpoonright i}(x)\land \phi_{\tau\upharpoonright i}(x)$ is inconsistent. We assign to the root of the tree the formula $\phi(x):=\exists y_1,y_2 \big( x= f(y_1,y_2)\big)$. If at the next level we go "left" then we make $y_1$ decomposable and $y_2$ indecomposable and if we go "right" we reverse the roles of $y_1, y_2$. More formally 
$$\phi_0(x):=\exists y_1,y_2 \Big(\exists y_3,y_4 \big(y_1=f(y_3,y_4)\land N(y_2)\land x= f(y_1,y_2)\big)\Big) $$ and 
$$\phi_1(x):=\exists y_1,y_2 \Big(\exists y_3,y_4 \big(y_2=(f(y_3,y_4)\land N(y_1)\land x= f(y_1,y_2)\big)\Big)$$. 

It is clear that each formula $\phi_0(x), \phi_1(x)$ is consistent with $\phi(x)$, but, by applications of Axioms $B_{ij}$ and Axiom $A_0$ togehter with the definition of indecomposability $\phi_0(x)\land\phi_1(x)$ is inconsistent.  We continue splitting the decomposable existential variables, i.e. $y_1$ in $\phi_0$ and $y_2$ in $\phi_1$, in the same manner. Thus, 
$$\phi_{00}(x):=\exists y_1,y_2 \Bigg(\exists y_3,y_4 \Big( \exists y_5, y_6 \big(y_3=f(y_5,y_6)\land N(y_4)\land y_1=f(y_3,y_4)\land N(y_2)\land x= f(y_1,y_2)\big)\Big)\Bigg) $$

$$\phi_{01}(x):=\exists y_1,y_2 \Bigg(\exists y_3,y_4 \Big( \exists y_5, y_6 \big(y_4=f(y_5,y_6)\land N(y_3)\land y_1=f(y_3,y_4)\land N(y_2)\land x= f(y_1,y_2)\big)\Big)\Bigg) $$

$$\phi_{10}(x):=\exists y_1,y_2 \Bigg(\exists y_3,y_4 \Big( \exists y_5, y_6 \big(y_3=f(y_5,y_6)\land N(y_4)\land y_2=f(y_3,y_4)\land N(y_1)\land x= f(y_1,y_2)\big)\Big)\Bigg) $$
and 
$$\phi_{11}(x):=\exists y_1,y_2 \Bigg(\exists y_3,y_4 \Big( \exists y_5, y_6 \big(y_4=f(y_5,y_6)\land N(y_3)\land y_2=f(y_3,y_4)\land N(y_1)\land x= f(y_1,y_2)\big)\Big)\Bigg) $$

Now, $\{\phi(x), \phi_0(x), \phi_{0i}(x)\}$ and $\{\phi(x), \phi_1(x), \phi_{1i}(x)\}$ are consistent for $i\in\{0,1\}$, but $\phi_{00}(x)\land\phi_{01}(x)$ and $\phi_{10}(x)\land\phi_{11}(x)$ are both inconsistent. 

We continue following the same pattern. The formula $\phi_{\eta\upharpoonright i}(x)$, for $\eta\in 2^{\omega}$, is complicated to write down, hence we will omit the general form. It is not hard, though, to verify the properties of the tree. 
\end{proof}

\begin{remark}
We will deal with the first-order theory, $\mathcal{T}_0$, of locally free algebras without indecomposable elements in the next section. We will prove that it does not admit a prime model, therefore it also has uncountably many types over the empty set. 
\end{remark}

The next result is an immediate corollary of Proposition \ref{UncountableTypes} (and Corollary \ref{UncountablyT0}).

\begin{corollary}
The first-order theory of any term algebra does not admit a countable saturated model. 
\end{corollary}

We now pass to the notion of model completeness.

\begin{definition}
    Let $T$ be first-order theory in a language $\mathcal{L}$. Then,
    \begin{enumerate}
        \item $T$ is said to be \emph{model complete} if for every $\mathcal{A},\mathcal{B}\models T$ such that $\mathcal{A}\subseteq \mathcal{B}$, then $\mathcal{A}$ is an elementary substructure of $\mathcal{B}$; 
        \item an $\mathcal{L}$-theory $T'$ is a \emph{model companion} of $T$ if it is model complete and every model of $T$ embeds in a model of $T'$, and conversely;
        \item an $\mathcal{L}$-theory $T'$ is a \emph{model completion} of $T$ if it is a model companion such that for every $\mathcal{A}\models T$ the $\mathcal{L}_{\mathcal{A}}$-theory $T'\cup\mathrm{diag}(\mathcal{A})$ is complete.
    \end{enumerate}
\end{definition}

As a motivating example, we recall that the theory of algebraically closed fields is the model completion of the theory of fields. For more details we refer the reader to \cite[Section 3.5]{ChangKeisler1990ModelTheory}.

\begin{proposition}
Let $1\leq \kappa\leq \omega$. Then the first order theory of the term algebra of rank $\kappa$, $Th(\mathcal{F}_\kappa)$, is not model complete.
\end{proposition}
\begin{proof}
Let $\mathcal{F}_\kappa$ be a term algebra of rank $\kappa$. Let $\{a_i \ | \ i<\kappa\}$ be the family of indecomposable elements. Consider, the subalgebra, $\mathcal{A}$, generated by $\{a_i \ | \ 1< i < \kappa\} \cup \{f_0(a_1,\ldots, a_1)\}$. It is immediate that $\mathcal{A}$ is isomorphic to $\mathcal{F}_\kappa$. Moreover, $\mathcal{A}$ is not an elementary subalgebra since $\exists y (f_0(y,\ldots, y)=f_0(a_1,\ldots,a_1))$ is true in $\mathcal{F}_\kappa$, but not in $\mathcal{A}$. 
\end{proof}

\begin{proposition}
The first-order theory of locally free algebras without indecomposable elements is model complete. 
\end{proposition}
\begin{proof}
It is enough to prove that modulo $\mathcal{T}_0$ any first-order formula in $\Sigma$ is equivalent to an existential one. Since $\mathcal{T}_0$ is complete, any sentence $\tau$ will be equivalent either to $\exists x(x=x)$, if it is implied by the theory, or equivalent to $\exists x(x\neq x)$, if its negation is implied. 

By Maltsev's quantifier elimination, Fact \ref{MalcevQE}, it is enough to prove that a standard formula is equivalent modulo $\mathcal{T}_0$ to an existential formula. Consequently, it is enough to show that formulas of the form $N_i(y)$ are equivalent modulo $\mathcal{T}_0$ to existential formulas. Indeed, each such formula is equivalent to $\bigvee_{j\neq i}\lnot N_j(y)$, which is existential.   
\end{proof}

\begin{corollary}
The theory of locally free algebras admits a model companion. 
\end{corollary}
\begin{proof}
The theory of locally free algebras and $\mathcal{T}_0$ are co-theories. Indeed, any locally free algebra embeds into a locally free algebra with no indecomposable elements as follows. If $\mathcal{A}$ is a locally free algebra with indecomposable elements, then we can use compactness to find a locally free algebra $\mathcal{B}_1$ which extends $\mathcal{A}$ and all indecomposable elements in $\mathcal{A}$ become decomposable in $\mathcal{B}_1$. Repeating the process we obtain a chain, $\mathcal{B}_1\subseteq \mathcal{B}_2\subseteq \ldots\subseteq \mathcal{B}_m\subseteq\ldots$, of locally free algebras each of which "decomposes" the indecomposable elements of the previous. It follows that its union is locally free and does not have indecomposable elements.   
\end{proof}

As a matter of fact the theory $\mathcal{T}_0$ is the model companion of each $Th(\mathcal{F}_\kappa)$, for $1\leq \kappa\leq \omega$. On the other hand, it is not the model completion. 

\begin{proposition}
    The first-order theory of locally free algebras does not admit a model completion.
\end{proposition}
\begin{proof}
By \cite[Proposition\cspace3.5.18]{ChangKeisler1990ModelTheory}, $\mathcal{T}_0$ is the model completion of the theory $T$ of locally free algebras if and only if $\mathcal{T}$ has the amalgamation property. 

If $\mathcal{A},\mathcal{B}$ are models of $\mathcal{T}$, and $e$ is the indivisible generator of $\mathcal{F}_1$, there exist two embeddings $\alpha:\mathcal{F}_1\rightarrow\mathcal{A}$ and $\beta:\mathcal{F}_1\rightarrow\mathcal{B}$ such that 
    \begin{equation*}
        \mathcal{A}\models \exists x\left(\alpha(e)=t(x,x)\right)\quad\text{and}\quad\mathcal{B}\models\exists y,z\left(y\neq z \wedge\beta(e)=t(y,z)\right),
    \end{equation*}
    \noindent
    with $t(x,y)$ an arbitrary $\Sigma$-term in two variables. However, the axiom scheme $(A_i)$ implies that there exist no model $\mathcal{C}$ and embeddings $\alpha':\mathcal{A}\rightarrow \mathcal{C}$, $\beta':\mathcal{B}\rightarrow\mathcal{C}$ such that $\alpha'(\alpha(e))=\beta'(\beta(e))$ in $\mathcal{C}$. This concludes the proof.
\end{proof}

\begin{remark}
It is not hard to see that if $\Sigma$ contains a unique unary function symbol, then $\mathcal{T}_0$ is the model completion of the theory of locally free algebras. 
\end{remark}

\section{Stability}
In this section we will deal with certain notions stemming out mostly from Shelah's classifiaction program. The most prominent such notion is stability. For the rest of the introduction we fix a complete first-order theory $T$. 

\begin{definition}
Let $\phi(\bar x, \bar y)$ be a first-order formula. We say that $\phi(\bar x, \bar y)$ has the order property in $T$, if there is a model $\mathcal{M}\models T$ and sequences of tuples $(\bar a_i)_{i<\omega}$, $(\bar b_i)_{i<\omega}$ from $\mathcal{M}$ such that 
\begin{equation} \mathcal{M}\models \phi(\bar a_i, \bar b_j) \ \ \ \ \textrm{if and only if} \ \ \ \ i<j
\label{order}
\end{equation}

Moreover, $T$ is stable if no formula $\phi(\bar x, \bar y)$ has the order property in $T$.
\end{definition}

\begin{fact}\ 
\begin{itemize}
\item The first-order theory $T$ is stable if and only if no formula of the form $\phi(x, \bar y)$ (where $x$ is a singleton) has the order property. 
\item The first-order formula $\phi(\bar x, \bar y)$ does not have the order property in $T$ if and only if there is $m<\omega$ such that whenever two sequences of tuples from any model of $T$ satisfy (\ref{order}), then their length is bounded by $m$.
\end{itemize}
\end{fact}

Conspicuous examples of natural stable theories are the theory of any $R$-module over any ring $R$, the theory of algebraically closed fields of any fixed characteristic and the theory of nonabelian free groups (\cite{Sela2013DiophantineVIIIStability}). 

There are various strengthenings of stability that have been studied along different lines of research. We recall that stability implies the existence of an independence relation, called forking independence, and at the level of formulas is defined as follows. We say that $\phi(\bar x, \bar b)$ forks over $A$, if in some elementary extension $\mathcal{N}$, that contains $A$ and $\bar b$, there exist infinitely many automorphisms $(f_i)_{i<\omega}\in Aut(\mathcal{N}/A)$ that fix $A$ and such that the set of images of $\phi(\mathcal{N},\bar b)$ under $(f_i)_{i<\omega}$ is $k$-inconsistent for some $k<\omega$, i.e. the intersection of any $k$ distinct conjugates of the set is empty. 

A type $p\in S(B)$ forks over $A\subseteq B$, if it contains a formula that forks over $A$. Any type $p\in S(A)$ has at least one non-forking extension over any set $B$ that contains $A$, i.e. there is $q\in S(B)$, such that $q\upharpoonright A=p$ and $q$ does not fork over $A$. A type $p\in S(B)$ is called stationary if it has, over any set of parameters containing $B$, a unique non-forking extension. A Morley sequence (of length $\omega$) of a stationary type $p\in S(B)$ is a sequence of realizations of non-forking extensions over the previous realizations, i.e. $a_{i+1}\models p_i$, where $p_i$ is the forking extension of $p$ over $B, a_1, \ldots, a_i$. A Morley sequence is essentially unique in the sense that any two such sequences have the same type over $B$. In addition, a stationary type admits a canonical base which is the smallest definably closed sets (in the sense of eq) over which the type is definable. The canonical base is always contained in the algabraic closure of a Morley sequence. 

\begin{definition}
A stable theory is $1$-based if the canonical base of any stationary type is contained in the algebraic closure (in the sense of eq) of any of its realizations. 
\end{definition}

\begin{definition}[Superstability]
A stable theory $T$ is superstable if no formula admits an infinite strictly descending forking chain $$\phi_0>\phi_1>\ldots>\phi_n>\ldots$$
where $\phi>\psi$, if $T\models \psi\rightarrow\phi$, $\phi$ is defined over $A$ and $\psi$ forks over $A$.
\end{definition}

\begin{definition}[Normality]
A first-order formula $\phi(\bar x, \bar y)$ is normal (with respect to $\bar x$) in $T$ if in any model $\mathcal{M}\models T$ and $\bar b$ a tuple from $\mathcal{M}$, any conjugate of $\phi(\mathcal{M}, \bar b)$ by an automorphism of $\mathcal{M}$ is either equal to $\phi(\mathcal{M}, \bar b)$ or disjoint from it. 

Moreover $T$ is normal if every formula is equivalent (modulo $T$) to a Boolean combination of normal formulas.
\end{definition}

The theory of any $R$-module over any ring $R$ is normal. This follows easily from the quantifier elimination down to positive primitive formulas and the understanding of their solutions sets - $\phi(\bar x, \bar b)$ is a coset of the submodule defined by $\phi(\bar x, \bar 0)$ in any model $\mathcal{M}$ of the theory. On the other hand, no theory of an algebraically closed field of a fixed characteristic is normal. 

\begin{definition}[Finite Cover Property]
A first-order formula $\phi(\bar x, \bar y)$ does not have the finite cover property in $T$ if there is a natural number $n<\omega$ such that whenever for any set of instances any subset of \ $\leq n$ instances is consistent then the whole set of instances is consistent.  

Moreover $T$ does not have the finite cover property if every formula $\phi(\bar x, \bar y)$ does not have the finite cover property.
\end{definition}

\begin{fact}
The first-order theory $T$ does not have the finite cover property if and only if every formula $\phi(x,\bar y)$ (where $x$ is a singleton) does not have the finite cover property.
\end{fact}

The theory of an equivalence relation with arbitrarily large finite classes is an example of a normal theory (hence stable) that has the finite cover property.  

In the following subsection we will explore the complete theories of locally free algebras with respect to the above notions. 

\subsection{Stability of locally free algebras}

We start by proving the non-superstability of any completion of the theory of locally free algebras. Recall that we have a blanket assumption of the existence of at least one function symbol of arity $\geq 2$. Our proof generalizes to any theory in which a "1-1" $n$-ary function, for $n\geq 2$, is definable.

\begin{definition}
Let $\mathcal{L}$ be a first-order signature (not necessarily functional). Let $\mathcal{M}$ be an $\mathcal{L}$-structure and $f:\mathcal{M}^n\rightarrow \mathcal{M}$ a definable function over $\mathcal{M}$. Let $p(x)\in S_1(\mathcal{M})$, a (complete) type over $\mathcal{M}$. We say that $f$ is "1-1" over $p(x)$, if there is a sequence of realizations $(a_i)_{i<\omega}$ of $p(x)$ such that $\models \forall \bar x\bar y\big(f(a_i,\bar x)\neq f(a_j,\bar y)\big)$, for every $i\neq j$.   
\end{definition}

\begin{theorem}\label{thm - general criterion non-superstability}
Let $\mathcal{L}$ be a first-order signature (not necessarily functional). Let $\mathcal{M}$ be a stable $\mathcal{L}$-structure and suppose there exists a definable function $f:\mathcal{M}^n\rightarrow \mathcal{M}$, for some $n\geq 2$. Suppose $f$ is "1-1" over a Morley sequence of a non-algebraic type over $\mathcal{M}$. Then, the first-order theory of $\mathcal{M}$ is not superstable. 
\end{theorem}
\begin{proof}
Let $(a_i)_{i<\omega}$ be the Morley sequence given by the hypothesis. We consider the formula, $\phi_0(x):=\exists \bar y\big(x=f(\bar y)\big)$. We claim that the formula $\phi_0$ is not superstable. Indeed, $\phi_1(x):= \exists y_2, \ldots, y_n \big(x=f(a_1, y_2, \ldots, y_n)\big)$ implies $\phi_0(x)$, is definable over $\mathcal{M}\cup\{a_1\}$ and forks over $\mathcal{M}$. Likewise, $\phi_{n+1}(x):=\exists y_2, \ldots, y_n\big(x=f(a_{n+1},f(a_{n},\ldots(f(a_1,y_2, \ldots, y_n))\ldots)\big)$ implies $\phi_n(x)$, is definable over $\mathcal{M}\cup\{a_1^1, \ldots, a_{n+1}^1\}$ and forks over $\mathcal{M}\cup \{a_1, \ldots, a_n\}$. This finishes the proof. 
 \end{proof}

The following result is now an immediate corollary of Theorem \ref{thm - general criterion non-superstability}. 

\begin{corollary}
Any complete theory of locally free algebras is not superstable.
\end{corollary}
\begin{proof}
By our blanket assumption there is a function symbol $f$ of arity $\geq 2$. It is enough to apply the previous theorem to $f$ and any Morley sequence of a non-algebraic type over a locally free algebra. 
\end{proof}

We next prove a strengthening of stability. 

\begin{theorem}\label{NormalProof}
Any complete theory of locally free algebras is normal.
\end{theorem}

We first collect a few auxiliary results needed for the proof of the theorem. We will follow Belegradek \cite{Belegradek1988LocallyFree} and work in a definitional expansion of the theory of locally free algebras. In particular, for every function symbol $f_i$ of arity $n_i$ and every $1\leq j\leq n_i$ we introduce a new unary function symbol, denoted $\pi_j^{f_i}$, that it should be understood as the projection of its argument in the $j$-th coordinate if the argument is $f_i$-decomposable or the identity otherwise. More formally: 

\[
\pi^{f_i}_j(x) \;=\;
\begin{cases}
x_j, & \text{if } x = f_i(x_1,\dots,x_{n_i}) \text{ for some } x_1,\dots,x_{n_i},\\[4pt]
x, & \text{otherwise.}
\end{cases}
\]

We call the signature expanded by these new symbols $\Sigma^*$. We prove that the theory of locally free algebras extended by the explicit definitions for each new function symbol, denoted by $T_{\Sigma^*}$, admits (full) quantifier elimination after moving to any completion. We first prove two preparatory lemmas. 

\begin{lemma}\label{Equalities}
Let $t,s$ be $\Sigma$-terms and $\bar b$ a tuple from a locally free algebra $\mathcal{M}$. Let $\phi(y):=\exists \bar z\big(y=t(\bar z, \bar b)\big)$ and $\psi(y):=\exists \bar z\big(y=s(\bar z, \bar b)\big)$ Then either $\phi(y)\land\psi(y)$ is inconsistent or there exists a $\Sigma$-term $r$ such that 
\[\mathcal{M}\models \big(\phi(y)\land \psi(y)\big) \leftrightarrow \exists \bar z\big(y=r(\bar z, \bar b)\big)\]
\end{lemma}
\begin{proof}
Assume that the conjunction $\phi(y)\land\psi(y)$ is consistent. The proof proceeds by induction on the maximum height of $t$ and $s$. If the maximum height is $0$, then obviously the formula that works is either $y=b_i$ for some $b_i$ from $\bar b$, or $\exists z\big(y=z\big)$. Now assume the result is true for any terms of maximum height less than $n$, we show it holds for terms of maximum height less than $n+1$. We may assume that $t$ is $f_0$-decomposable. Since $\phi(y)\land\psi(y)$ is consistent we must have that $s$ is $f_0$-decomposable as well. Now if $\phi(y)=\exists \bar z\big(y=f_0(t_1, \ldots, t_{n_0})\big)$ and $\psi(y)=\exists \bar z\big(y=f_0(s_1, \ldots, s_{n_0})\big)$, we must have  coordinate wise consistency, i.e. $\exists\bar z\big(x_i=t_i(\bar z, \bar b)\big)\land \exists\bar z\big(x_i=s_i(\bar z, \bar b)\big)$ are consistent for every $i\leq n_0$. Hence by the induction hypothesis, there exists $r_1, \ldots, r_{n_0}$ for which for every $i\leq n_0$: 
$$\mathcal{M}\models\Big(\exists\bar z\big(x_i=t_i(\bar z, \bar b)\big)\land \exists\bar z\big(x_i=s_i(\bar z, \bar b)\big)\Big)\leftrightarrow \exists \bar z\big(x_i=r_i(\bar z, \bar b)\big)$$  

In particular the formula $\exists \bar z\big(y=f_0(r_1(\bar z, \bar b), \ldots, r_{n_0}(\bar z, \bar b))\big)$ works.
\end{proof}

\begin{lemma}\label{Inequalities}
Let  $t, t_1, \ldots, t_m$ be $\Sigma$-terms and $\bar b$ a tuple from a locally free algebra $\mathcal{M}$. Let $\bigl\{ \exists \bar z \big( y=t(\bar z, \bar b)\big), \forall \bar z \big( y\neq t_1(\bar z, \bar b)\big), \ldots, \forall \bar z \big( y\neq t_m(\bar z, \bar b)\big)\bigr\}$ be a consistent set of formulas. 

Then there exist terms $s_1, \ldots, s_n$ and $\bar a$ from the union of the indecomposable elements of $\mathcal{M}$ (if they exist) and the $\Sigma$-subalgebra generated by $\bar b$ such that $\exists \bar w \big( y=t(s_1(\bar w), \ldots, s_n(\bar w), \bar b)\big)$ implies the conjunction of the above formulas.
\end{lemma}
\begin{proof}
The proof is by induction on the height of the term $t$. If $height(t)=0$, then $\exists \bar z \big( y=\bar z, \bar b)\big)$ is equivalent either to $y=y$ or to $y=b_i$ for some $b_i$ in $\bar b$. In the latter case the result follows trivially. In the former case, we first consider the $t_i$ terms that are $f_0$-decomposable. If non exists, then we set $\exists \bar z \big (y=f_0(\bar z)\big)$ and the result follows. If some exist and the conjunction of those formulas is not equivalent to $\forall \bar z\big(y\neq f_0(\bar z)\big)$, then we collect the negations of those formulas in the set $F_0(y):=\big\{\exists \bar z \big(y=t_{i_1}(\bar z, \bar b)\big), \exists \bar z \big(y=t_{i_2}(\bar z, \bar b)\big), \ldots,  \exists \bar z \big(y=t_{i_\ell}(\bar z, \bar b)\big)\big\}$. We prove that there exists a formula of the form $\exists \bar z \big(y=f_0(s_1(\bar z, \bar b), \ldots, s_{n_0}(\bar z, \bar b))\big)$ that works, i.e. it implies the conjunction of $\lnot F(y)$, by induction on the maximal height of the arguments inside the function symbol $f_0$.

For height $0$, we have that every such formula is equivalent to $\exists \bar z \big(y=f_0(\bar z, \bar b)\big)$ in which at least one $b_i$ from $\bar b$ appears. Then the formula $\exists \bar z \,\bigl(y = f_0(\bar z, \bar c)\bigr)$, where, in every occurrence of a parameter $b_i$ in any formula from $F_0(y)$, we substitute an element $c_i$ from the $\Sigma$-subalgebra generated by $\bar b$ chosen to be distinct from all finitely many elements appearing in the same position, yields the desired conclusion.

Suppose, the result is true for any terms of height at most $n$, we show it is true for terms of height at most $n+1$. By the hypothesis, there must be some tuple $\bar d$ such that $\mathcal{M}\models \bigwedge\lnot F_0(f_0(\bar d)))$. We fix those entries that are indecomposable elements, and we may assume that for at least one co-ordinate the element of this co-ordinate is not indecomposable. Without loss of generality suppose the first $i$ co-ordinates of the tuple $\bar d$ are indecomposable elements and the rest are not. We keep those formulas from $\F_0(y)$ which have the same indecomposable elements as $\bar d$ in the corresponding co-ordinates or just a variable, we must have at least one such formula. We call this subset $F_0^1(y)\subseteq F_0(y)$. We continue refining $F_0^1(y)$ by keeping the formulas with which there exists some $\bar z$ that gives the value $d_{i+1}$ in the $i+1$-co-ordinate of $f_0$, if no such formula exists then if $d+1$ is $f_j$-decomposable, we consider those formulas whose $i+1$-th term is $f_j$-decomposable, by the induction hypothesis there exists a term $r$ whose solution set lives in the complement of the union of the solution sets of the aforementioned terms, hence choosing one formula from $F_0^1(y)$ and changing its $i+1$-th coordinate to this term works. We can now proceed up to the last co-ordinate and we claim that the repeatedly refined set of formulas does not contain any formula whose last co-ordinate term can take the value $d_{n_0}$. Indeed, then we would have that $\bar d$ satisfies $F_0(y)$. Hence the argument we used for the $i+1$-th co-ordinate, still applies here and we get the result. 

We thus may assume that the inequalities imply that $y$ is indecomposable and possibly not equal to some list of indecomposable elements (this list cannot be exhaustive). Then the formula that works is $y=c$, for $c$ some indecomposable element not in the previous list of indecomposable elements.  

To complete the induction assume that we have the result for any term $t$ of height at most $n$ and we want to prove it for terms of height $n+1$. Without loss of generality $t$ is $f_0$-decomposable. We consider again the set $F_0(y)$, if it is empty we are done. We know that there exists $\bar d$ such that  $f_0(\bar d)$ satisfies $\bigwedge\lnot F_0(y)$ and $\exists \bar z\big(y=t(\bar z, \bar b)\big)$. Let's fix for notational convenience the term $t:=f_0(s_1, \ldots, s_{n_0})$. Some of those coordinates correspond to indecomposable elements and as before we may assume that the first $i$ coordinates have indecomposable elements while the rest do not. For at least one coordinate the corresponding element is decomposable,  otherwise the formula $y=f_0(d_1, \ldots, d_{n_0})$ for the tuple of indecomposable elements works. Now, refine $F_0(y)$ to the subset $F_0^1(y)$ by keeping the formulas that have the same indecomposable elements in the corresponding coordinates or just a variable. We must have at least one such formula, otherwise the formula $\exists \bar z\big(y=f_0(d_1, \ldots, d_i, s_{i+1}, \ldots, s_{n_0})\big)$  works. We keep refining $F_0^1(y)$ by keeping the formulas that for some value for $\bar z$ give the value $d_{i+1}$ to their $i+1$-coordinate. Notice that this value for $\bar z$ must be consistent with fixing the value of the first $i$ coordinates, one needs to be careful because we allowed not just indecomposable elements but also single variables to appear - the values of these variables will be fixed to the value of the indecomposable element. Again at least one such formula must exist, otherwise by the induction hypothesis (since the term $s_{i+1}$ has height at most $n$) we can find a term that once plugged in in the $i+1$-th coordinate of $t$ in the place of $s_{i+1}$ gives us what we want. By repeatedly applying this argument we reach at the last coordinate for which indeed there cannot be formulas and values for $\bar z$ that yield the element $d_{n_0}$, hence we get the term and formula we wanted.

\end{proof}

\begin{theorem}
Any completion of the first-order theory $T_\Sigma^*$ admits (full) quantifier elimination.  
\end{theorem}

\begin{proof}
We first fix a completion of $T_\Sigma^*$ that we call $T^*$. Recall that this amounts to deciding the number of indecomposable elements. 
We will prove that if $\mathcal{M}, \mathcal{N}\models T^*$ and $\mathcal{A}$ is any $\Sigma^*$-substructure of $\mathcal{M}, \mathcal{N}$, then for any quantifier free formula $\phi(\bar x, y)$, and for any $\bar a$ from $\mathcal{A}$ if $\mathcal{M}\models \phi(\bar a, b)$, then there is $c\in \mathcal{N}$ such that $\mathcal{N}\models \phi(\bar a, c)$. 

We observe that any atomic formula $t_1(\bar x, y)=t_2(\bar x, y)$ for $t_1, t_2$ $\Sigma^*$-terms is equivalent to conjunctions of atomic formulas of the form $\pi_{i_1}(x_\nu)=\pi_{j_1}(x_\mu), \ldots, \pi_{i_n}(x_\nu)=\pi_{j_n}(y)$, where $\pi_k$'s are compositions of functions symbols in $\Sigma^*\setminus \Sigma$, e.g. $\pi_1(x):=\pi_3^{f_0}(\pi_2^{f_1}(x))$. 

We may assume that the quantifier free formula $\phi(\bar x, y)$ consists of conjunctions of formulas of the form $\pi_{i}(x_\nu)=\pi_{j}(y)$ and negations of such formulas. Suppose $\mathcal{M}\models \phi(\bar a, b)$ for some $\bar a$ from $\mathcal{A}$ and $b$ in $\mathcal{M}$. Each formula of the form $\pi_{1}(a_i)=\pi_{2}(y)$ defines a disjunction of formulas of the form $\exists \bar z \big(y=t_1(\bar z, \pi_1(a_i))\big)$ for some $\Sigma$-term $t_1$ that depends on $\pi_2$ (notice that $t_1$ is a term in the original signature $\Sigma$ and $\pi_1(a_i)$ is an element in $\mathcal{A}$). We consider all possible combinations of choices of one disjunctive clause from each conjunction. Some of those combinations will be consistent. Each consistent combination is equivalent, by Lemma \ref{Equalities}, to an atomic formula of the form  $\exists \bar z \big(y=t(z_1, \ldots, z_m, \pi_{1_\ell}(a_1), \ldots,\pi_{n_\ell}(a_n))\big)$, where $t(z_1, \ldots, z_m, \pi_{1_\ell}(a_1), \ldots,\pi_{n_\ell}(a_n))$ is a $\Sigma$-term. 

Any solution (in $\mathcal{M}$) of such formula satisfies all conjunctive clauses of equations. Moreover, any $b$ that satisfies the conjunction of equations, satisfies one of these formulas. We collect these formulas in a set, which we call $E(y)$. Apparently any $m$-tuple from $\mathcal{A}$ can be plugged in into $t$ and provide a suitable solution that lives in $\mathcal{A}$.  If each term $t$ is a term in $\pi_{i_1}(a_1), \ldots,\pi_{i_n}(a_n)$ alone we are done. Hence we may assume that some $z_j$ appears in some term $t$. 

We now deal with the inequalities. By the same argument each inequality is equivalent to conjunctions of formulas of the form $\forall \bar z \big(y\neq t_1(\bar z, \pi_1(a_i))\big)$, where $t_1$ is a $\Sigma$-term. These conjunctions, which we call $N(y)$, cannot imply the negation of each formula in $E(y)$, otherwise $\phi(\bar a, y)$ would be inconsistent. In particular, we must have for at least one formula in $E(y)$, say $\exists \bar z \big(y=t(z_1, \ldots, z_m, \pi_{1_\ell}(a_1), \ldots,\pi_{n_\ell}(a_n))\big)$ that this formula is consistent with $N(y)$. By Lemma \ref{Inequalities} there exists a formula of the form $\exists \bar z \big(y=r(\bar z, \bar d)\big)$ where $\bar d$ is a tuple from the union of the set of indecomposable elements of $\mathcal{M}$ and $\mathcal{A}$. Since $\mathcal{M}, \mathcal{N}\models T^*$ we always have an adequate amount of indecomposable elements in $\mathcal{N}$ to replace the indecomposable elements of $\mathcal{M}$ in $\exists \bar z \big(y=r(\bar z, \bar d)\big)$. It is not hard to see that any solution of the later formula (with the indecomposable elements replaced) in $\mathcal{N}$ yields an element $c$ in $\mathcal{N}$ for which $\mathcal{N}\models \phi(\bar a, c)$.
\end{proof}

We now may prove Theorem \ref{NormalProof}. 
\begin{proof}[Proof of Theorem \ref{NormalProof}]
It is enough to prove that every $\Sigma^*$-formula $\phi(\bar x, \bar y)$ is a Boolean combination of normal formulas. Indeed, by quantifier elimination, every $\Sigma^*$-formula is equivalent to Boolean combinations of formulas of the form $\pi_1(x_i)=\pi_2(y_j)$ and each such formula is normal.
\end{proof}

It follows from \cite{HrushovskiPillay1987WeaklyNormal} that:

\begin{corollary}
Any complete theory of locally free algebras is stable and $1$-based.
\end{corollary}

We may also prove that the theory of locally free algebras with no indecomposable elements does not admit a prime model. 

\begin{proposition}
The first-order theory $\mathcal{T}_0$ does not have a prime model.
\end{proposition}
\begin{proof}[Sketch]
Let $a$ be an element of a locally free algebra $\mathcal{M}$ without inndecomposable elements. We will show that $tp^{\mathcal{M}}(a)$ is not isolated. Indeed, by quantifier elimination all formulas in $tp^{\mathcal{M}}(a)$ have the form $\pi_1(x)=\pi_2(x)$ and their negations. We consider all formulas in the type, where the composition of projections $p_i^{f_j}$ on both sides of the equalities and inequalities have length at most $n$. These formulas describe some relations between the arguments of skeletons of height at most $n$. We observe that we can always find two elements that have identical relations with $a$ up to $n$-skeleton, but disagree on the $n+1$-skeleton. Hence, $tp^{\mathcal{M}}(a)$ is not isolated.    
\end{proof}

\begin{corollary}\label{UncountablyT0}
The first-order theory $\mathcal{T}_0$ has uncountably many types over the empty set.
\end{corollary}

The notion of elimination of imaginaries have been introduced as a meaningful tool in model theory by Poizat in \cite{poizat1983galois}. For the different notions of elimination of imaginaries and their connections with existence of canonical bases we refer the reader to \cite{CasanovasFarre2004WeakEI}.

Before proving a number of preparatory lemmas concerning the algebraic-closure operator in locally free algebras, we introduce the following notation. Let $A$ be a subset of a locally free algebra. The $\Sigma$-closure of $A$, denoted $\Sigma\text{-}\mathrm{cl}(A)$, is the union of the increasing chain
\[
A_0:=A,\qquad 
A_{i+1}:=\{\, f_j(\bar a)\ :\ \bar a\in A_i^{n_j}\, \& \, f_j\in \Sigma\}\ \ \text{for } i\in\mathbb{N},
\]
Thus
\[
\Sigma\text{-}\mathrm{cl}(A)\;=\;\bigcup_{i\in\mathbb{N}} A_i.
\]
If $A\neq\emptyset$, then $\Sigma\text{-}\mathrm{cl}(A)$ is a subalgebra of the ambient algebra, and in particular it is itself locally free.

We also define the {\em projection closure} of $A\subset \mathcal{M}$ as 
\[
\operatorname{pr}(A)\ :=\ \bigl\{\pi(a)\ :\ a\in A\ \text{and}\ \pi \text{ is a finite composition of projection maps } \pi_i^{f_j}\bigr\}.
\]
By convention $pr(\emptyset)=\emptyset$ and $\Sigma\text{-}\mathrm{cl}(\emptyset)=\emptyset$
\begin{lemma}\label{algclosureind}
Let $\mathcal{M}$ be a locally free algebra with no indecomposable elements, and let $A\subseteq \mathcal{M}$. 
Then the (real-sort) algebraic closure of $A$ coincides with the (real-sort) definable closure of $A$, and moreover they are both equal to the $\Sigma$-closure of $pr(A)$.
\end{lemma}
\begin{proof}
We treat the case $A=\emptyset$ separately. We will prove that any type over the empty set has infinitely many solutions in an enough saturated model. This implies that $dcl(\emptyset)=acl(\emptyset)=\emptyset$ By quantifier elimination, a complete type $p(x)$, is determined by (in)decomposability conditions and equalities or inequalities between elements of all possible depths. Let $a$ be a realization of this type in an elementary extension, $\mathcal{N}$, of $\mathcal{M}$. Consider the set of formulas $q(x):=p(x)\cup \{x\neq a\}$. Then an easy compactness argument shows that $q(x)$ is consistent. Indeed, for every finite subset $\Pi(x)$ of $p(x)$, there are infinitely many elements that satisfy it (since there are no indecomposable elements), hence $\Pi(x)\cup\{x\neq a\}$ is consistent. Thus, we may find $b\neq a$, that satisfies $p(x)$. Repeating the argument we can find more than $n$ realizations for every $n$, thus infinitely many.   

We next assume that $A$ is not empty. 
First we observe that the $\Sigma\text{-}\mathrm{cl}(pr(A))$ is a subset of $dcl(A)$. If $\mathcal{M}$ has no indecomposable element, then $\Sigma-cl(pr(A))$ is a model of $Th(\mathcal{M})$, hence it is definably closed and in particular $dcl(A)=acl(A)=\Sigma\text{-}\mathrm{cl}(prA)$. 
\end{proof}

\begin{lemma}\label{algclosurefin}
Let $\mathcal{M}$ be a locally free algebra with only finitely many indecomposable elements, and let $A\subseteq \mathcal{M}$.
Then the (real-sort) algebraic closure of $A$ is the $\Sigma$-closure of
$\operatorname{Ind}(\mathcal{M}) \,\cup\, \operatorname{pr}(A)$,
where $\operatorname{Ind}(\mathcal{M})$ denotes the set of indecomposable elements of $\mathcal{M}$.

\end{lemma}
\begin{proof}
If $\mathcal{M}$ has finitely many indecomposable elements, then $\operatorname{Ind}(\mathcal{M})$ is a subset of $acl(\emptyset)$. In particular the $\Sigma$-closure of $\operatorname{Ind}(\mathcal{M}) \,\cup\, \operatorname{pr}(A)$ is a subset of $acl(A)$. Since $\operatorname{Ind}(\mathcal{M}) \,\cup\, \operatorname{pr}(A)$ is not empty its $\Sigma$-closure is a locally free (sub)-algebra of $\mathcal{M}$ and in addition it has the same number of indecomposable elements as $\mathcal{M}$. Indeed, it contains $\operatorname{Ind}(\mathcal{M})$, and if any other element in the closure is indecomposable, then it must belong to $\operatorname{pr}(A)\setminus \operatorname{Ind}(\mathcal{M})$, but then it is decomposable in $\mathcal{M}$ and in particular the elements that decompose it belong to $\operatorname{Ind}(\mathcal{M}) \,\cup\, \operatorname{pr}(A)$, a contradiction. 

\end{proof}

Note that in this case if there are more than one (but finitely many)  indecomposable elements  $dcl(\emptyset)\neq acl(\emptyset)$. Indeed, it is not hard to check that $dcl(\emptyset)=\emptyset$, while $acl(\emptyset)$ is the prime model of the theory. 

The following lemma is less trivial. 

\begin{lemma}\label{algclosureinf}
Let $\mathcal{M}$ be a locally free algebra with infinitely many  indecomposable elements, and let $A\subseteq \mathcal{M}$.
Then the (real-sort) algebraic closure of $A$ coincides with the $\Sigma$-closure of $pr(A)$.
\end{lemma}
\begin{proof}
The case where $A=\emptyset$ is similar to the proof of Lemma \ref{algclosureind}. The only difference is that in this case we may have formulas implying the indecomposability of some projections of $x$, but because there are infinitely many such indecomposable elements, the realizations of each complete type over the empty set still form an infinite set.

We may assume that $pr(A)$ does not contain infinitely many indecomposable elements, otherwise the result follows as in the proof of Lemma \ref{algclosurefin}. We will show that if $c$ does not belong to $\Sigma\text{-}\mathrm{cl}(pr(A))$, then $p(x):=tp(c/A)$ has infinitely many realizations. Indeed, for each $n$, the $n$-skeleton of $c$ would have at least one elements that does not belong to $\Sigma\text{-}\mathrm{cl}(pr(A))$. Consider the set of formulas $q(x):=p(x)\cup \{x\neq c\}$. Then every finite subset $\Pi(x)\subset p(x)$, has infinitely many solutions, hence $\Pi(x)\cup \{x\neq c\}$ is consistent, and by compactness $q(x)$ is consistent. Repeating the argument we can find more than $n$ realizations for every $n$, thus infinitely many.  
\end{proof}

The following result is an easy consequence of the previous lemmas. 

\begin{corollary}\label{algclosurecap}
Let $\mathcal{M}$ be a locally free algebra and $A, B\subset \mathcal{M}$ be algebraically closed subsets. Then $A\cap B$ is algebraically closed.  
\end{corollary}
\begin{proof}
If $A\cap B=\emptyset$, then this can occur only in the cases where the algebra has either infinitely many indecomposable elements or none at all. In either case one has $acl(\emptyset)=\emptyset$. Henceforth assume that $A\cap B\neq\emptyset$. Since both $A$ and $B$ are closed under projections, their intersection is also closed under projections, and therefore
\[
\operatorname{pr}(A\cap B)=A\cap B.
\]
Likewise, as $A$ and $B$ are $\Sigma$-closed, the set $A\cap B$ is $\Sigma$-closed. Consequently,
\[
acl(A\cap B)\;=\;\Sigma\text{-}\mathrm{cl}\bigl(\operatorname{pr}(A\cap B)\bigr)\;=\;A\cap B,
\]
as desired.
 
\end{proof}

\begin{theorem}\label{WEI}
Any completion of the theory of locally free algebras weakly eliminates imaginaries.
\end{theorem}

\begin{proof}
We show that every definable set $D$ admits a smallest algebraically closed set of parameters over which it is definable. It suffices to prove the following claim: if $A$ and $B$ are algebraically closed sets and $D$ is definable over both $A$ and $B$, then $D$ is definable over $A\cap B$. Indeed, by Corollary~\ref{algclosurecap} the intersection $A\cap B$ is algebraically closed, and hence taking intersections yields the desired smallest algebraically closed set. 

We may assume that the formula that defines $D$ over $A$ is equivalent to a disjunction of conjunction of atomic and negation of atomic formulas, $\bigvee\bigwedge \phi(\bar x, \bar a)$, where each disjunctive clause is not contained in the union of the rest. An element from $A\setminus B$ must appear in at least one conjunctive clause. In addition, we may assume that in all disjunctive clauses the formulas that contain elements from $A\setminus B$ are not implied by the rest. Indeed, we may remove the disjunctive clauses that are definable over $A\cap B$, the remaining set is still definable over both $A$ and $B$.
We now take cases: 
\begin{itemize}

\item suppose in one of the disjunctive clauses we have an atomic formula $\pi(x_i)=a$ for some $a\in A\setminus B$. Since $B$ is algebraically closed and $a\notin B$. We have that $a$ has infinitely many images under automorphisms that fix $B$. In particular $D$ is not definable over $B$, since the automorphisms fixing $B$ move $\pi(x_i)=a$ to infinitely many pairwise disjoint sets.  

\item suppose elements $a\in A\setminus B$ only appear in negations of atomic formulas $\pi(x_i)\neq a$. We consider the complement of $D$, which is still definable over both $A$ and $B$. It is a conjunction of disjunctions $\bigwedge\bigvee \lnot\phi(\bar x, \bar a)$. In each conjunctive clause there is an atomic formula $\pi(x_i)=a$, where $a\in A\setminus B$, which is not contained in the disjunction of formulas over $A\cap B$ that belong to the same conjunctive clause. If any tuple that belongs to $D^{c}$ belongs to the intersection of disjunctions definable over $A\cap B$, then $D^{c}$ is definable over $A\cap B$. If not, then there is a formula $\pi(x_i)=a$, where $a\in A\setminus B$ for which at least one solution belongs to a tuple of the solution set $D^{c}$. This formula has infinitely many pairwise disjoint images under automorphisms that fix $B$ pointwise. As before, $D^{c}$ is not definable over $B$.     

\end{itemize}

\end{proof}

%

\begin{theorem}
Any complete theory of locally free algebras has trivial forking
\end{theorem}
\begin{proof}
Let $\bar a, \bar b, \bar c$ be finite tuples from a locally free algebra $\mathcal{M}$ and $A\subseteq \mathcal{M}$. Assume $\bar a, \bar b , \bar c$ are pairwise independent over $A$, we will show that $\bar a$ is independent from $\bar b, \bar c$ over $A$. If not, by $1$-basedness, there exists an element in $acl^{eq}(A, \bar a)\cap acl^{eq}(A, \bar b, \bar c)\setminus acl^{eq}(A)$ and by the weak elimination of imaginaries we have that there exists an element in the real sort $acl(A, \bar a)\cap acl(A, \bar b, \bar c)\setminus acl(A)$. By the characterization of the algebraic closure operator we must have that there exists an element in  $pr(A,\bar a)\cap pr(A,\bar b, \bar c)\setminus pr(A)$. We next observe that $pr(D)=\bigcup\{pr(d) \ : \ d\in D\}$. Hence, there is come $a_i$ from $\bar a$ not in $A$ and $b_j$ from $\bar b$ not in $A$ (or $c_j$ from $\bar c$ not in $A$), such that $\pi_1(a_i)=\pi_2(b_j)$ (or $\pi_1(a_i)=\pi_2(c_j)$). In particular, $\bar a$ either forks with $\bar b$ or with $\bar c$ over $A$, a contradiction.   
\end{proof}

\begin{corollary}
No infinite group is interpretable in any model of any comlete theory of a locally free algebra. 
\end{corollary}

\begin{theorem}\label{NFCP}
No complete theory of locally free algebras has the finite cover property.
\end{theorem}
\begin{proof}
An easy Ramsey type argument shows that nfcp is closed under disjunctions. Hence, by quantifier elimination,  we only need to prove that finite conjunctions of formulas of the form $\pi_1(x)=\pi_2(x), q_1(x)=q_2(y_j)$ and negations have nfcp. We may assume that we don't have conjunctive clauses that only involve parameters $\bar y$, moreover in any conjunctive clause that some $y$ appears we may replace the composition of projections $\pi(y)$ by a fresh variable. Hence we may assume that the conjunctions have the following form: $$\bigwedge \pi_i(x)=\pi_j(x)\bigwedge \pi_q(x)\neq \pi_p(x) \land \pi_1(x)=y_1\land \pi_2(x)=y_2\land\ldots\land\pi_m(x)=y_m\land $$ $$ \land \pi_{m+1}(x)\neq y_{m+1}\land\ldots\land\pi_k(x)\neq y_k$$ 

We observe that if a sequence of $n+1$ tuples witnesses fcp, i.e. the intersection of the $n+1$ instances is inconsistent, but the intersection of any $n$ is consistent, then the first $m$ values of the tuples must be fixed, as they correspond to the variables $y_1, \dots, y_m$, that render the formula inconsistent if they are different. Hence, we may decompose the above conjunctions into two formulas $\phi^-(x,y_{m+1}, \ldots, y_k):=\pi_{m+1}(x)\neq y_{m+1}\land\ldots\land\pi_k(x)\neq y_k$ and $\phi^+(x,y_1, \ldots, y_m)$. By the previous observation, for any tuples $\bar b_0, \ldots, \bar b_n$ that witness fcp, the solution set of $\phi^{+}(x,\bar b_i)$ is fixed.  Since $\phi(x,\bar b_0)\land\ldots\land\phi(x, \bar b_n)$ is inconsistent, we have that for each $a$ in $\phi^+(x)$, there is a $m< j \leq k$ and $0\leq i\leq n$, such that $a$ is in $\pi_j(x)=b_i^j$. The fixed set $\phi^{+}(x)$ is the union of formulas of the form $\psi_t(x):=\exists \bar x (x=t(\bar x, b_0^1, \ldots, b_0^m)\bigwedge \pi_i(x)=\pi_j(x)\bigwedge \pi_q(x)\neq \pi_p(x)$. There is a bound on the number of disjunctive clauses in the union that only depends on $\pi_i$, $i\leq m$, and not on the parameter $\bar b_0$. We call this bound $M_{\phi^+}$. The atomic formulas $\pi_i(x)=\pi_j(x)$ and their negations, impose some equality/inequality relations and (in)decomposability conditions between elements among $x, \bar x$. 

We observe that the solution set of $\pi_j(x)=b_i^j$ cannot contain the solution set $\pi_q(x)=b_0^q$ for any $q\leq m$, otherwise a single instance of the formula would be inconsistent. In particular, a disjunct corresponding to $\pi_j(x)=b_i^j$ has the form $\exists \bar w \big(x=t(\ldots, t_\mu(\ldots,b_i^j,\ldots), \ldots)\big)$, where $t_\mu(\ldots,b_i^j,\ldots)$ is at the place of some $x_\mu$ from $\bar x$. Hence, the only possibility for all elements $a$, which are solutions of $\psi_t(x)$, to be covered is that the following conditions hold: $x_\mu$ is indecomposable, the theory has exactly $n+1$ indecomposable elements and $\pi_j(x)=b_0^j\land\ldots\land \pi_j(x)=b_n^j$ covers them. 

Hence, for any theory with finitely many indecomposable elements, say $n$, if for a formula as above and any set of instances if every $(n+1)\times M_{\phi^+}\times |\bar x|$ subset is consistent, then the whole set is consistent. In the first-order theories, $\mathcal{T}_\omega$ and $\mathcal{T}_0$ we may take the bound to be  $2\times M_{\phi^+}\times |\bar x|$.  
\end{proof}

\bibliography{biblio}
\newpage 
\ \\
{\bf Davide Carolillo}\\ \\
Dipartimento di Matematica ``Giuseppe Peano''\\
Universit\`a degli Studi di Torino\\
Palazzo Campana\\
Via Carlo Alberto 10\\
10123 Torino\\
Italy
\ \\ \\
{\bf Yifan Jia}\\ \\
University of Electronic Science and Technology of China (UESTC) \\
Qingshuihe Campus of UESTC, No.2006, Xiyuan Avenue, West Hi-tech Zone,\\
Chengdu, Sichuan, Zip/Postal Code: 611731 \\ 
P.R.China
\ \\ \\
{\bf Bakh Khoussainov}\\ \\
University of Electronic Science and Technology of China (UESTC) \\
Qingshuihe Campus of UESTC, No.2006, Xiyuan Avenue, West Hi-tech Zone,\\
Chengdu, Sichuan, Zip/Postal Code: 611731 \\ 
P.R.China
\ \\ \\
{\bf Rizos Sklinos}\\ \\
Academy of Mathematics \& Systems Science\\
Chinese Academy of Science\\
Room 1016 SiYuan Building\\
No.55 of Zhongguancun East Road\\
Haidian District, Beijing Zip/Postal Code: 100190\\
P. R. China

\end{document}